\documentclass{amsart}
\usepackage[utf8x]{inputenc}

\usepackage{bbm}
\usepackage{hyperref}

\newtheorem{bigthm}{Theorem}
\newtheorem{bigthmp}{Theorem}

\newtheorem{theorem}{Theorem}[section]
\newtheorem{lemma}[theorem]{Lemma}
\newtheorem{proposition}[theorem]{Proposition}

\newtheorem*{theorem*}{Theorem}
\newtheorem{corollary}[theorem]{Corollary}
\newtheorem{conjecture}{Conjecture}

\theoremstyle{definition}
\newtheorem{definition}[theorem]{Definition}
\newtheorem{example}{Example}

\theoremstyle{remark}

\newtheorem*{remark}{Remark}

\numberwithin{equation}{section}
\newcommand{\norm}[3]{\ensuremath{\left\Vert#1\right\Vert_{#2}^{#3}}}

\newcommand{\indicator}{{\mathord{\mathbbmss 1}}}
\newcommand{\Z}{\mathbb Z}

\newcommand{\N}{\mathbb N}
\newcommand{\R}{\mathbb R}
\newcommand{\T}{\mathbb T}

\newcommand{\e}{\varepsilon}
\newcommand{\al}{\alpha}
\newcommand{\be}{\beta}
\newcommand{\ga}{\gamma}

\begin{document}

\title{Powers of sequences and convergence of ergodic averages}

 \author{N. Frantzikinakis}
\address[Nikos  Frantzikinakis]{Department of Mathematics\\
 University of Memphis\\
  Memphis, TN \\ 38152 \\ USA } \email{frantzikinakis@gmail.com}

\author{M. Johnson}
\address[Michael Johnson]{Department of Mathematics \& Statistics\\
 Swarthmore College\\
  Swarthmore, PA \\ 19081 \\ USA } \email{mjohnso3@swarthmore.edu}

\author{E. Lesigne}
\address[Emmanuel Lesigne]{Laboratoire de Math\'ematiques et Physique Th\'eorique (UMR CNRS 6083)\\
Universit\'e Fran\c{c}ois Rabelais Tours \\
F\'ed\'eration de Recherche Denis Poisson\\
Parc de Grandmont \\
37200 Tours\\ France} \email{emmanuel.lesigne@lmpt.univ-tours.fr}

\author{M. Wierdl}
\address[M\'at\'e Wierdl]{Department of Mathematics\\
  University of Memphis\\
  Memphis, TN \\ 38152 \\ USA } \email{wierdlmate@gmail.com}

\begin{abstract}
A sequence $(s_n)$ of integers is good for the mean ergodic theorem if for each invertible measure preserving system $(X,\mathcal{B},\mu,T)$ and any bounded measurable function $f$, the averages   $ \frac1N \sum_{n=1}^N f(T^{s_n}x)$ converge in the $L^2(\mu)$ norm.
We construct a sequence $(s_n)$ that is good for the mean ergodic theorem, but the sequence  $(s_n^2)$ is not. Furthermore, we show that for any set of bad exponents $B$, there is a sequence $(s_n)$ where $(s_n^k)$ is good for the mean ergodic theorem  exactly when $k$ is not in $B$. We then extend this result to multiple ergodic averages of the form $ \frac1N \sum_{n=1}^N f_1(T^{s_n} x)f_2(T^{2s_n}x)\ldots f_\ell(T^{\ell s_n}x).$  We also prove a similar result for pointwise convergence of single ergodic averages.
\end{abstract}

\thanks{The  first and last authors were supported
 by NSF grants DMS-0701027 and   DMS-0801316.}

\maketitle

\setcounter{tocdepth}{1}
\tableofcontents

\section{Introduction}
\label{sec:introduction}

\subsection{Main result}
\label{sec:main-result}
It is well known that for any fixed positive integer $k$ the sequence
$1^k,2^k,3^k\dots$ is
good for the mean ergodic theorem. This means that for every measure preserving system
 and function $f$ in $L^2(\mu)$, the averages
\begin{equation}
  \label{eq:1}
  \frac1N \sum_{n=1}^N f(T^{n^k}x)
\end{equation}
converge in the $L^2(\mu)$ norm as $N\to\infty$. Using the spectral theorem for unitary operators, this is equivalent to
the convergence of the averages
\begin{equation}
  \label{eq:2}
  \frac1N \sum_{n=1}^N e^{2\pi i n^k\alpha}
\end{equation}
as $N\to\infty$ for any real number $\alpha$. 

An illustrative question for our paper is the following.  Is there a
sequence $s_1,s_2,s_3\dots$ of positive integers so that the averages
\begin{equation}
  \label{eq:3}
  \frac1N \sum_{n=1}^N e^{2\pi i s_n\alpha}
\end{equation}
converge for any real number $\alpha$, but for some $\alpha$, the
averages
\begin{equation}
  \label{eq:4}
  \frac1N \sum_{n=1}^N e^{2\pi i s_n^2\alpha}
\end{equation}
do not converge as $N\to\infty$? In other words, we ask if there is a sequence
$s_1,s_2,s_3\dots$ of positive integers which is  good for
the mean ergodic theorem, but the sequence of squares
$s_1^2,s_2^2,s_3^2\dots$ of the sequence is not  good for
the mean ergodic theorem?

Similarly, we can ask: is there a sequence $s_1,s_2,s_3\dots$ of
positive integers which is \emph{not}  good for the mean
ergodic theorem, but the sequence of squares $s_1^2,s_2^2,s_3^2\dots$
of the sequence \emph{is} good for the mean ergodic
theorem?

Perhaps surprisingly, the answer to \emph{both} questions is
\emph{yes}, indicating that the convergence properties of positive powers of a sequence are independent of those of the original sequence.  In fact, in this paper we prove
the following result showing the total independence of powers of a sequence
for the mean ergodic theorem.

\begin{bigthm}
  \label{sec:main-result-1}
  Let $B$ be an arbitrary set  of  positive integers.
  Then there exists an increasing  sequence $s_1,s_2,s_3\dots$ of positive integers
  such that
  \begin{itemize}
  \item The sequence $s_1^g,s_2^g,s_3^g\dots$ is good for the mean
    ergodic theorem for any ``good'' exponent $g\in \mathbb N
    \setminus B$.
  \item The sequence $s_1^b,s_2^b,s_3^b\dots$ is not good for the mean
    ergodic theorem for any ``bad'' exponent $b\in B$.
  \end{itemize}
\end{bigthm}

Using the spectral theorem for unitary operators, we get the following equivalent formulation of our theorem:
\begin{bigthmp}
  \label{sec:main-result-2}
  Let $B$ be an arbitrary set of positive integers.  Then there exists
  an increasing sequence $s_1,s_2,s_3\dots $ of positive integers such that
  \begin{itemize}
  \item For $g\in \mathbb N \setminus B$, the averages
    \begin{equation}
      \label{eq:5}
      \frac1N \sum_{n=1}^N e^{2\pi i s_n^g\alpha}
    \end{equation}
    converge as $N\to\infty$ for any real number $\alpha$.
  \item For $b\in B$, there exists a real number $\alpha$ such that the
    averages
    \begin{equation}
      \label{eq:6}
      \frac1N \sum_{n=1}^N e^{2\pi i s_n^b\alpha}
    \end{equation}
    do not converge  as $N\to\infty$.
  \end{itemize}
\end{bigthmp}

Similar results related to issues of recurrence were proved in
 \cite{FLW2}.  The original motivation to search for results which
express the independence of powers of a sequence for various
properties, comes from the papers of Deshouillers, Erd\H os, S\' ark\"ozy (\cite{DES})
and Deshouillers, Fouvry (\cite{DF}).  In these papers,
the authors prove results analogous to ours but for bases of the
positive integers.

In our paper, we generalize Theorem~\ref{sec:main-result-1} to
multiple ergodic averages and prove a version for pointwise
convergence of single ergodic averages.  We state these generalizations in the next subsection,
where we also give precise definitions of the concepts used throughout
the paper.

\subsection{Definitions and generalizations}
\label{sec:defin-gener}
All along the article we shall use the word  {\it system},
 or the term {\it measure preserving system},
 to designate a quadruple $(X,\mathcal{B},\mu,T)$, where $T$ is an {\it invertible}
 measure preserving transformation of a probability space $(X,\mathcal{B},\mu)$.
 By insisting that our system is invertible, we make sure that $T^k$ is well defined
 for negative integers $k$.


\begin{definition}
  Let $\ell$ be a positive integer, $(s_n)$ be a sequence
  of integers, and  $(X,\mathcal{B},\mu, T)$ be a system.

  We say that the sequence of integers $(s_n)$ is \emph{good for
    $\ell$-convergence for the system} $(X,\mathcal{B},\mu,T)$ if for any
  bounded, measurable functions $f_1,f_2,\dots,f_\ell$, the averages
  \begin{equation*}
    \frac1N \sum_{n=1}^N f_1(T^{s_n}x)\cdot f_2(T^{2s_n}x)\cdot \ldots
    \cdot f_\ell(T^{\ell s_n}x)
  \end{equation*}
  converge in the $L^2(\mu)$ norm as $N\to\infty$.

We say the sequence $(s_n)$ is \emph{universally good for $\ell$-convergence} if it is good for  $\ell$-convergence for any system $(X,\mathcal{B},\mu,T)$. We  often abbreviate this by saying that ``$(s_n)$ is good for $\ell$-convergence" and  refer to the case $\ell=1$ as \emph{single mean  convergence}.
\end{definition}

We refer the reader to  \cite{RW} for examples of sequences that are good for single mean convergence.
Examples of sequences that are  good for $\ell$-convergence for every $\ell\in \N$, include $s_n=n$, shown by Host and Kra (\cite{HK1}) and later by Ziegler (\cite{Z}), as well as $s_n=p(n)$ where $p(n)$ is an integer polynomial, as shown by Host and Kra (\cite{HK2}) and Leibman (\cite{L1}).

We give a strengthening of Theorem~A that
 related to problems of $\ell$-convergence:
\begin{theorem}\label{T:multiple}
  Let $B$ be an arbitrary set of positive integers  and $\ell\in \N$.
  Then there exists an increasing  sequence  $(s_n)$ of positive integers such that
  \begin{itemize}
  \item For every $g\in \N\setminus B$, the sequence $(s_n^g)$ is good
    for  $\ell$-convergence.
  \item For every $b\in B$, the sequence $(s_n^b)$ is not good for
   $\ell$-convergence.
  \end{itemize}
\end{theorem}
Next we introduce a notion related
to pointwise convergence of ergodic averages.
\begin{definition}
  Let $(s_n)$ be a sequence of integers, and let
  $(X,\mathcal{B},\mu, T)$ be a system.

  We say that the sequence $(s_n)$ is \emph{good for  the pointwise
   ergodic theorem for the system $(X,\mathcal{B},\mu,T)$} if for any $f\in
  L^2(\mu)$, the averages
  \begin{equation*}
    \frac1N \sum_{n=1}^N f(T^{s_n}x)
  \end{equation*}
  converge as $N\to\infty$ for almost every $x$.

We say the sequence $(s_n)$ is \emph{universally good for the pointwise ergodic theorem} if it is good for the pointwise ergodic theorem for any system $(X,\mathcal{B},\mu,T)$. We often abbreviate this by saying that ``$(s_n)$ is good for the pointwise ergodic theorem."
\end{definition}
\begin{remark}
The preceding definition
 is given for the class of functions in $L^2(\mu)$. It is known  that similar  definitions for the class of functions in $L^p(\mu)$  gives rise
to different notions for different values of $p\in [1,+\infty]$  (see [RW, Chapter VII]). However, for the sequences that we construct in the present paper (and which have  positive density) the properties of being good for the class of functions in $L^p(\mu)$ are equivalent for all $p>1$. For $p=1$, our arguments do not apply, the reason being that there is no strong maximal inequality along the sequence $(n^k)$ in $L^1(\mu)$.
\end{remark}

Notice that any sequence that is good for the pointwise ergodic theorem is also good for the mean ergodic theorem.
Various examples of sequences that are good for the pointwise ergodic theorem are known
(see e.g. \cite{RW}).  A particular example that will be used later is the case of a sequence  $(p(n))$ where $p$ is any polynomial with integer coefficients. This case was treated by Bourgain in  \cite{Bou1}.

We give a strengthening of Theorem~A
 related to problems of pointwise convergence:
\begin{theorem}\label{T:main}
 Let $B$ be an arbitrary set of positive integers. Then there exists an
  increasing sequence $(s_n)$ of positive integers such that
  \begin{itemize}
  \item For every $g\in \N\setminus B$, the sequence $(s_n^g)$ is good
    for the pointwise ergodic theorem.
  \item For every $b\in B$, the sequence $(s_n^b)$ is not good for
    the mean ergodic theorem.
  \end{itemize}
\end{theorem}

\subsection{Further remarks and conjectures}
\label{sec:furth-remarks-conj}
If for a given $\ell\in\mathbb{N}$, the sequence $(s_n)$ is good for  $\ell$-convergence,
then, as Theorem~\ref{T:multiple} shows, we cannot assert in general
that any of its powers $(s_n^k)$, where $k\geq 2$, is good for
$\ell$-convergence.  In  contrast with this, we expect the
following result to be true:
\begin{conjecture}\label{C:everyl}
  Suppose that the sequence of integers  $(s_n)$ is good for  $\ell$-convergence for every $\ell\in \N$.
Then  for every $k\in\N$, the sequence $(s_n^k)$ is good for
  $\ell$-convergence for every $\ell\in \N$.
\end{conjecture}
To support this conjecture, let us mention that if the sequence
$(s_n)$ is good for $2$-convergence then the sequence $(s_n^2)$ is
good for single mean  convergence (see Lemma~\ref{L:FLW} below). In fact,
Conjecture~\ref{C:everyl} would be true if the following more general
statement holds:
\begin{conjecture}\label{C:kl}
  Suppose that the sequence of integers $(s_n)$ is good for mean $(k\ell)$-convergence. Then the sequence $(s_n^k)$ is good for
  $\ell$-convergence.
\end{conjecture}
The conjecture holds for $\ell=1$. This is shown in \cite{FLW1} and a key ingredient of the proof is
 the spectral theorem for unitary operators which gives
convenient necessary and sufficient conditions for  single mean convergence. We currently do not have such a convenient
characterization for $\ell$-convergence when $\ell\geq 2$. The
following conjecture would fill this gap if true:
\begin{conjecture}\label{C:iff}
  Let $(s_n)$ be a sequence of integers. The following three
  statements are equivalent:
  \begin{itemize}
  \item The sequence $(s_n)$ is good for $\ell$-convergence.

  \item The sequence $(s_n)$ is good for $\ell$-convergence
  for every $\ell$-step nilsystem.\footnote{If $G$ is an
      $\ell$-step nilpotent Lie group and $\Gamma$ is a discrete
      cocompact subgroup, then the homogeneous space $X = G/\Gamma$ is
      called an {\it $\ell$-step nilmanifold}.  If $T_a(g\Gamma) = (ag)
      \Gamma$ for some $a\in G$, $\mathcal{X}$ is the Borel
      $\sigma$-algebra of $X$, and $m$ is the Haar measure on $X$,
      then the system $(X, \mathcal{X}, m, T_{a})$ is called an
      $\ell$-step nilsystem.}
  \item For every $\ell$-step nilmanifold $X=G/\Gamma$, every $a\in
    G$, and $f\in C(X)$, the sequence $\big(\frac{1}{N}\sum_{n=1}^N
    f(a^{s_n}\Gamma)\big)$ converges as $N\to+\infty$.
  \end{itemize}
\end{conjecture} We are mainly interested in knowing if the third (or second) condition implies the first.
In the case that the set $S=\{s_n, n\in \N\}$ has positive upper density, this implication follows immediately from the nilsequence
decomposition result of Bergelson, Host, and Kra  (Theorem 1.9 in  \cite{BHK}).

For general sequences $(s_n)$, using the spectral theorem for unitary operators, we can  verify Conjecture~\ref{C:iff}  for $\ell=1$.  It is possible to see that if
Conjecture~\ref{C:iff} is true then Conjecture~\ref{C:kl} is also true
(and hence Conjecture~\ref{C:everyl}).

\bigskip
{\bf Notation:} The following notation will be used throughout the
article: $Tf=f\circ T$, $e(t)=e^{2\pi i t}$.

\bigskip
{\bf Acknowledgment.}
The authors wish to thank the Mathematical Sciences Research Institute in Berkeley for providing partial support to complete work on the present article during its special semester program on Ergodic Theory and Additive Combinatorics.

\section{Single mean convergence}\label{sectionmean}
In this section we shall prove Theorem~\ref{sec:main-result-1}.  This will help us illustrate the main ideas
behind
the more complicated arguments we shall use in the proof of Theorem~\ref{T:multiple} and of Theorem~\ref{T:main}
(which both extend Theorem~\ref{sec:main-result-1}).

We shall prove our theorem in its equivalent formulation of
Theorem~\ref{sec:main-result-2}.

Let $B$ be a fixed set of positive integers (possibly empty), and let $\alpha$ be a fixed irrational number.  It is sufficient to find a sequence $(s_n)$ satisfying the following two conditions:
\begin{enumerate}
\item[(s1)] For every $b\in B$, the sequence $\Big(\frac1N
  \sum_{n=1}^N e(s_n^b\alpha)\Big)$ diverges.
\end{enumerate}

\begin{enumerate}
\item[(s2)] For every $g\in \N\setminus B$ and every $\beta\in\R$, the
  sequence $\Big(\frac1N \sum_{n=1}^N e(s_n^g\beta)\Big)$ converges.
\end{enumerate}
The rest of this section will be devoted to the construction of a
sequence $(s_n)$ that satisfies conditions (s1) and (s2).

\subsection{Definition of the sequence  $(s_n)$}\label{SS:s_n1}
 We denote
\begin{equation}\label{E:I}
  I_+=\left\{ x\in \mathbb T\mid  \, \cos(2\pi x)\ge \sqrt 2/2\right\}\quad\text{and}\quad
  I_-=\left\{ x\in \mathbb T\mid  \, \cos(2\pi x)\le -\sqrt 2/2\right\}.
\end{equation}
These are intervals of length $1/4$ on the torus.  The sequence
$(s_n)$ consists of the elements of a set $S$, taken in increasing
order, that is defined as follows:
\begin{equation}\label{E:S}
  S=\bigcup_{j\ge 1}\left\{ n\in\N\mid [2^{2j-1}\leq n< 2^{2j}
    ,\, \
    n^{b_j}\alpha\in I_+]\text{ or }[2^{2j}\leq  n< 2^{2j+1},\, n^{b_j}\alpha\in I_- ]\right\}
\end{equation}
for some appropriately chosen sequence $(b_j)$ of elements of $B$.  We shall construct a sequence $(b_j)$ so that:

$\bullet$ Every element of $B$ appears infinitely often in the
sequence $(b_j)$. This guarantees  that condition (s1) holds.

$\bullet$ The first appearance of elements of $B$ in the sequence
$(b_j)$ happens late enough (with respect to $j$) to guarantee that
certain equidistribution properties are satisfied.  All unspecified
elements $b_j$ will be set to be equal to some fixed element of
$B$. This will enable us to verify condition (s2).

Let us now state explicitly the properties that the sequence $(b_j)$
is going to satisfy:
\begin{enumerate}
\item[(b1)]  $b_j\in B$, and all elements of $B$ appear infinitely
  often in the sequence $(b_j)$.

\item[(b2)] We have
$$  \lim_{j\to\infty}\sup_{N>2^{2j-1}}\left|\frac1N\sum_{n=1}^N \indicator_{I_\pm}(n^{b_j}\alpha) -\frac14
\right|=0.$$

\item[(b3)] For every nonzero $l\in\Z$, $\beta \in \R$, and $g\in G$, we have
$$
\lim_{j\to\infty}\sup_{N\geq2^{2j-1}} \left|\frac1N \sum_{n=1}^{N} e(l
  n^{b_j}\alpha + n^g\beta)\right|=0.
$$
  \end{enumerate}
  In the next subsection we construct such a sequence  $(b_j)$.

  \subsection{Construction of the sequence $(b_j)$}\label{SS:b_j1}
  The following lemma will be essential for our construction:
  \begin{lemma}\label{uniform}
    Let $b,g\in\N$ with $b> g$ and $\alpha\in\R\setminus
    \mathbb{Q}$. Then
    \begin{equation*}
     \lim_{N\to\infty} \frac1N \sum_{n=1}^{N} \indicator_{I_\pm}(n^b\alpha)
      = \frac14, \text{ and}
    \end{equation*}
    \begin{equation*}
     \lim_{N\to\infty} \sup_{\beta\in\R} \Big|\frac1N \sum_{n=1}^{N} e( n^b\alpha + n^g\beta)\Big| = 0.
    \end{equation*}
  \end{lemma}
  \begin{proof} The first part follows from Weyl's equidistribution
    theorem. The second part is a direct consequence of van der
    Corput's classical inequality (see e.g. \cite{KN74}). Applying it $g$ times, we can
    estimate  the trigonometric sums  by a quantity that converges to $0$ as $N\to \infty$ uniformly in $\beta$. This completes the proof.
  \end{proof}
  Note that, for each fixed $b$ only finitely many $g$'s are involved
  in the lemma. Therefore, the convergence is also uniform in $g$.

\begin{proposition}\label{construction}
  There exists a sequence $(b_j)$ that satisfies conditions (b1),
  (b2), and (b3) of Section~\ref{SS:s_n1}.
\end{proposition}
\begin{proof}
  Using Lemma~\ref{uniform}, we get that for every $b\in\N$, $l\in \N$, and $\varepsilon>0$, there
  exists $J=J(b,l,\varepsilon)\geq1$ that satisfies
  \begin{equation}\label{E:11}
    \sup_{N\geq2^{2J-1}}\  \left|\frac1N \sum_{n=1}^{N} \indicator_{I_\pm}(n^b\alpha) -\frac14
    \right|\leq\varepsilon;
  \end{equation}
  and such that
  \begin{equation}\label{E:22}
    \sup_{N\geq2^{2J-1}}\ \sup_{\beta\in\R}\ \sup_{g\in\N,\,g<b} \left|\frac1N \sum_{n=1}^{N} e(l n^b\alpha + n^g\beta)\right|\leq\varepsilon.
  \end{equation}
  Furthermore, we can assume that $J=J(b,l,\varepsilon)$ is increasing
  with respect to the variables $b,l$, and decreasing with respect to
  the variable $\varepsilon$.

  We write $B=\{a_t, t\in \N\}$ where $a_1<a_2<\ldots$. We
  construct a sequence $(b_j)$ that satisfies the following conditions: (i)
  every integer $a_t$ appears infinitely often in the range of
  $(b_j)$, (ii) for $t\geq 2$, the first appearance of $a_t$ in $(b_j)$ happens at a
  time $j$ that is greater than $J(a_t,t,1/t)$, and (iii) all values of $b_j$
  that are left unspecified are set to be equal to $a_1$. Notice that
  condition (ii) guarantees that \eqref{E:11} and \eqref{E:22} hold
  for $b=b_j$ and $J=j$.

  The explicit construction goes as follows:

  Define $J_1=J(a_1,1,1)$ and $b_{J_1}=a_1$.

  Define $J_2=\max\{J(a_2,2,1/2),J_1+2\}$ and
  $b_{J_2}=a_2,\,b_{J_2+1}=a_1$.

  Inductively, we define $J_t=\max\{J(a_t,t,1/t),J_{t-1}+t\}$

  \hfill {and
    $b_{J_t}=a_t,\,b_{J_t+1}=a_{t-1},\ldots,b_{J_t+t-1}=a_{1}$.}

  We claim that this sequence $(b_j)$ has the advertised properties.  First, we note
  that every integer $a_t$ appears infinitely many times in the range
  of $(b_j)$, so condition (b1) holds.  Now let $\beta\in \R$, $g\in
  G$, and $l\in \N$.  Choose $j$ large enough so that $J_t\leq
  j<J_{t+1}$ for some $t\geq l$. Then, by construction, we have
  $b_j\in\{a_1,a_2,\ldots,a_t\}$. If $N\geq 2^{2j-1}$ then $N\geq
  2^{2J_t-1}$, and $J_t\geq J(a_k,l,1/t)$ for all $k$ between $1$ and
  $t$. It follows from \eqref{E:11} that
$$
\sup_{N\geq2^{2j-1}} \left|\frac1N \sum_{n=1}^{N}
  \indicator_{I_\pm}(n^{b_j}\alpha) -\frac14 \right|\leq\frac1t,
$$
and so condition (b2) holds.  Furthermore, it follows from
\eqref{E:22} that for every $b_j$ greater than $g$ we have
\begin{equation}\label{E:b_j}
  \sup_{N\geq2^{2j-1}} \left|\frac1N \sum_{n=1}^{N} e(l n^{b_j}\alpha + n^g\beta)\right|\leq\frac1t.
\end{equation}
It remains to deal with those $b_j$ that are less than $g$. Since
there are only finitely many values of the sequence $(b_j)$ that are
less than $g$, by Weyl's equidistribution theorem we have
      $$
      \lim_{N\to\infty}\sup_{j\in \N, b_j<g}\left| \frac1N \sum_{n=1}^{N} e(l
        n^{b_j}\alpha + n^g\beta)\right| = 0.$$ Combining this with \eqref{E:b_j} gives that
      condition (b3) also holds, completing the proof.
    \end{proof}


 \subsection{The sequence $(s_n)$ satisfies conditions (s1) and (s2)}
 The goal of this section is to complete the proof of Theorem~\ref{sec:main-result-2} by  proving the following proposition:
 \begin{proposition}\label{P:mean}
   Suppose that the sequence $(b_j)$ satisfies conditions (b1),(b2),
   and (b3) of Section~\ref{SS:b_j1}.  Then the sequence $(s_n)$
   defined by \eqref{E:S} satisfies conditions (s1) and (s2).
 \end{proposition}

 First we show that the set $S$ in \eqref{E:S} has positive density.
 \begin{lemma}\label{density}
   Suppose that the sequence $(b_j)$ satisfies condition (b2) of
   Section~\ref{SS:s_n1}. Then the set $S$ in \eqref{E:S} has density
   $1/4$.
 \end{lemma}
 \begin{proof}
   This is a direct  consequence of Lemma~\ref{L:Wierdl} in the
   Appendix.
 \end{proof}

 Since $S$ has positive density, condition (s1) is equivalent to
 \begin{equation}\label{M'1}
   \text{For every $b\in B$, the sequence}\ \Big(\frac1N
   \sum_{n=1}^{N} \indicator_S(n)e(n^b\alpha)\Big)\ \text{diverges,}
 \end{equation}
 and condition (s2) is equivalent to
 \begin{equation}\label{M'2}
   \text{For every $g\in \N\setminus B$ and $\beta\in\R$,}\text{ the sequence } \Big(\frac1N
   \sum_{n=1}^{N} \indicator_S(n)e(n^g\beta)\Big)\ \text{converges.}
 \end{equation}
 We first show that the conditions imposed on the sequence $(b_j)$
 guarantee that condition \eqref{M'1} (and as a result condition (s1))
 is satisfied by the set $S$.
 \begin{proposition}\label{P:1}
   Suppose that the sequence $(b_j)$ satisfies conditions (b1) and (b2)
   of Section~\ref{SS:s_n1}. Then the sequence $(s_n)$ defined by  \eqref{E:S} satisfies
   condition (s1).
 \end{proposition}
 \begin{proof}
   Fix $b\in B$. By condition $(b1)$ there are arbitrarily large
   values of $j$ for which $b_j=b$. For any such $j$ we have
   \begin{gather*}
     \frac1{2^{2j}}\sum_{n=1}^{2^{2j}}\indicator_S(n)\cos(2\pi n^b\alpha)-\frac1{2^{2j+1}}\sum_{n=1}^{2^{2j+1}}\indicator_S(n)\cos(2\pi n^b\alpha)=\\
     \frac1{2^{2j+1}}\sum_{n=1}^{2^{2j}}\indicator_S(n)\cos(2\pi
     n^b\alpha)-\frac1{2^{2j+1}}\sum_{n=2^{2j}+1}^{2^{2j+1}}\indicator_S(n)\cos(2\pi
     n^b\alpha)
     \geq\\
     \frac1{2^{2j+1}}\sum_{n=1}^{2^{2j-1}}\indicator_S(n)
     (-1)+\frac1{2^{2j+1}}\sum_{n=2^{2j-1}+1}^{2^{2j}}\indicator_S(n)
     (\sqrt2/2)-\frac1{2^{2j+1}}\sum_{n=2^{2j}+1}^{2^{2j+1}}\indicator_S(n)
     (-\sqrt2/2).
   \end{gather*}
   Using condition $(b2)$ and Lemma \ref{density}, we see that, for
   large $j$, the last quantity is near
   $-\frac1{16}+\frac{\sqrt2}{32}+\frac{\sqrt2}{16}$, which is
   positive. This shows that (\ref{M'1}) holds. Since the set $S$ has
   positive density (Lemma~\ref{density}), condition (s1) is also
   satisfied.
 \end{proof}
 Next we show that the conditions imposed on the sequence $(b_j)$
 guarantee that condition \eqref{M'2} (and hence condition (s2)) is satisfied. We first need two lemmas.
 \begin{lemma}\label{L:mean}
   Let $l,m$ be two nonzero integers, $\alpha\in \R\setminus
   \mathbb{Q}$, and suppose that the sequence $(b_j)$ satisfies
   condition (b3) of Section~\ref{SS:s_n1}. Let us define a sequence
   $(e_n)$ by
   \begin{equation*}
     e_n=
     \begin{cases}
       e(l n^{b_j}\alpha), \quad n\in [2^{2j-1}, 2^{2j}); \\
       e(m n^{b_j}\alpha), \quad n\in [2^{2j}, 2^{2j+1}).
     \end{cases}
   \end{equation*}
   Then for every $\beta\in \mathbb R$ and  $g\in \mathbb
   N\setminus B$, we have
   \begin{equation*}
    \lim_{N\to\infty} \frac1N \sum_{n=1}^{N} e_n \ \! e(n^g\beta) = 0 .   \end{equation*}
 \end{lemma}
 \begin{proof}
   Fix $\beta\in \mathbb R$ and $g\in \mathbb N\setminus B$.  By
   condition (b3) we have for every nonzero integer $k$ that
$$
\lim_{j\to\infty}\sup_{N>2^{2j-1}}\left|\frac1N \sum_{n=1}^{N} e(k
  n^{b_j}\alpha + n^g\beta)\right|=0.
$$
Applying this for $k=l$, and $k=m$, and using Lemma~\ref{L:Wierdl} in
the Appendix we  get the advertised result.
\end{proof}
Given a sequence $(b_j)$ of positive integers and functions
$\phi,\psi\colon \T \to \mathbb{C}$, define the sequence
\begin{equation}\label{alternance}
  f_n(\phi,\psi)=
  \begin{cases}
    \phi(n^{b_j}\alpha), \quad n\in [2^{2j-1}, 2^{2j}); \\
    \psi(n^{b_j}\alpha), \quad n\in [2^{2j}, 2^{2j+1}).
  \end{cases}
\end{equation}
\begin{lemma}\label{L:mean1}
  Let $\phi_0={\bf 1}_{I_+}-1/4, \psi_0={\bf 1}_{I_-}-1/4$, $\alpha\in \R\setminus \mathbb{Q}$, and suppose that the
  sequence $(b_j)$ satisfies condition (b3) of
  Section~\ref{SS:s_n1}.
  Then for every $\beta\in \mathbb R$ and $g\in \mathbb
  N\setminus B$, we have
  \begin{equation*}
   \lim_{N\to\infty} \frac1N \sum_{n=1}^{N} f_n(\phi_0,\psi_0) \ \! e(n^g\beta) = 0.
  \end{equation*}
\end{lemma}
\begin{proof}
  By Lemma \ref{L:mean}, the result is true if in place of $\phi_0$ and $\psi_0$ we use
  trigonometric polynomials. To complete the proof, we use a standard
  approximation argument. We give the details for the convenience of the reader.

    It suffices to
  show that for every $\varepsilon>0$ there exist trigonometric
  polynomials $\phi$ and $\psi$ with zero integral such that
  \begin{equation}\label{E:f_0}
    \limsup_{N\to\infty}\frac1N\sum_{n=1}^N\left|f_n\left(\phi_0,\psi_0\right)-f_n(\phi,\psi)\right|\leq\e.
  \end{equation}
  We can approximate in $L^1(\T)$ the function $\phi_0$ by a
  trigonometric polynomial. After composing with a translation this
  trigonometric polynomial will approximate the function $\psi_0$ as
  well.  So there exist two trigonometric polynomials $\phi$ and
  $\psi$, with zero integral, such that $
  \int_{\T}\left|\phi-\phi_0\right|=\int_{\T}\left|\psi-\psi_0\right|=\theta\leq\varepsilon.
  $ Notice that
  \begin{equation}
    \left|f_n\left(\phi_0,\psi_0\right)-f_n(\phi,\psi)\right|=
    f_n\left(\left|\phi-\phi_0\right|,\left|\psi-\psi_0\right|\right).
  \end{equation}
  So in order to establish \eqref{E:f_0} it suffices to show that
  \begin{equation}\label{tot}
    \limsup_{N\to\infty}\frac1N\sum_{n=1}^Nf_n(|\phi-\phi_0|,|\psi-\psi_0|)\leq \e.
  \end{equation}
   Let $\Phi_0=|\phi-\phi_0|$ and
  $\Psi_0=|\psi-\psi_0|$. Since both $\Phi_0$ and $\Psi_0$ are Riemann
  integrable and have integral $\theta$, the following
  holds: For every $\delta>0$ there exist four continuous
  functions $\phi_1,\phi_2,\psi_1,\psi_2$, with zero integral, such
  that
$$
\phi_1+\theta-\delta\leq\Phi_0\leq\phi_2+\theta+\delta\quad\text{and}\quad
\psi_1+\theta-\delta\leq\Psi_0\leq\psi_2+\theta+\delta.
$$
It follows that
\begin{equation}\label{E:delta}
  f_n(\phi_1,\psi_1) +\theta-\delta\leq f_n(\Phi_0,\Psi_0)\leq f_n(\phi_2,\psi_2) +\theta+\delta.
\end{equation}
Moreover, since
\begin{equation}\label{tard}
  \lim_{N\to\infty}\frac1N\sum_{n=1}^Nf_n(\phi,\psi)=0
\end{equation}
for every trigonometric polynomials $\phi$ and $\psi$ with zero
integral, by uniform approximation, this remains true if $\phi$ and
$\psi$ are continuous functions on the torus, with zero integral.
Thus, we deduce from \eqref{E:delta} and \eqref{tard} (applied to
$\phi=\phi_i$, $\psi=\psi_i$ for $i=1,2$), and the fact that $\delta$
was arbitrary, that
$$
\lim_{N\to\infty}\frac1N\sum_{n=1}^Nf_n(\Phi_0,\Psi_0)=\theta\leq \e.
$$
Hence, \eqref{tot} is established, completing the proof.
\end{proof}

\begin{proposition}\label{P:2}
  Suppose that the sequence $(b_j)$ satisfies conditions (b2) and (b3)
  of Section~\ref{SS:s_n1}. Then the sequence $(s_n)$ defined by
  \eqref{E:S} satisfies condition (s2).
\end{proposition}
\begin{proof}
  We apply Lemma \ref{L:mean1}. We get
  for every $\beta\in \mathbb R$ and  $g\in \mathbb N\setminus
  B$ that
  \begin{equation*}
    \lim_{N\to \infty} \frac1N \sum_{n=1}^{N} (\indicator_S(n)-1/4)\ \! e(n^g\beta) =0.
  \end{equation*}
 Hence, condition \eqref{M'2} is satisfied. Since the set $S$ has
  positive density (Lemma~\ref{density}), condition (s2) is also
  satisfied.
\end{proof}

Combining Propositions~\ref{P:1} and \ref{P:2} we deduce Proposition~\ref{P:mean}, completing the proof of Theorem~A'.

\section{Multiple mean convergence}
In this section we shall prove Theorem~\ref{T:multiple}. The argument
is similar to the one used to prove Theorem~\ref{sec:main-result-1},
but there are some extra complications since our analysis relies on
some more intricate multiple convergence results. To avoid repetition, we do
not give details of proofs that can be immediately extracted using
arguments of the previous section.
\subsection{A reduction}
We need one preliminary
result that  was proved in \cite{FLW1} in the special case
where the polynomial $p$ is a monomial. A very similar argument gives
the following more general result:
\begin{lemma}\label{L:FLW}
  Suppose that the sequence $(s_n)$ is good for
  $\ell$-convergence. Then for every polynomial $p\in\R[t]$ with $\deg p\leq
  \ell$,   the sequence  $\big(\frac{1}{N}\sum_{n=1}^N e(p(s_n))\big)$ converges.
\end{lemma}
\begin{example}
Let us illustrate how one proves Lemma~\ref{L:FLW} in the case where $\ell=2$.
Suppose that the sequence $(s_n)$ is good for
$2$-convergence. Let $p(t)= 2\alpha t^2 +\beta t +\gamma$ for some $\alpha,\beta,\gamma \in \R$.
We define the transformation
$R\colon\mathbb{T}^3\to\mathbb{T}^3$ by
$$
  R(t_1,t_2,t_3)=(t_1+\alpha,t_2+2t_1+\alpha, t_3+\beta),
  $$
and  for $k\in\Z$ the functions
$$
f_1(t_1,t_2,t_3)=e(k(-2t_2+t_3)), \quad f_2(t_1,t_2,t_3)=e(kt_2).
$$
Then $$
  R^n(t_1,t_2,t_3)=(t_1+n\alpha,t_2+2nt_1+n^2\alpha, t_3+n\beta).  $$
As a consequence, the averages
$$ \frac{1}{N}\sum_{n=1}^N R^{s_n}f_1\cdot R^{2 s_n} f_2=
e(k(t_3-t_2))\cdot \frac{1}{N} \sum_{n=1}^N e(k(2\alpha s_n^2+\beta s_n))  $$
converge as $N\to\infty$.
\end{example}
Using the previous lemma, we can  deduce Theorem~\ref{T:multiple} from
the following result that we shall prove next:
\begin{proposition}\label{mean'}
  Let $B$ be an arbitrary set of positive integers, $\ell\in \N$, and $\alpha\in \mathbb
  R\setminus\mathbb Q$. Then there exists an increasing sequence $(s_n)$
  of integers such that
  \begin{enumerate}
  \item[(s1)] For every $b\in B$, the sequence $\Big(\frac1N
    \sum_{n=1}^N e((s_n^{\ell b}+s_n^b)\alpha)\Big) $ diverges.
  \item[(s2)] For every $g\in \N\setminus B$, system
    $(X,\mathcal{B},\mu,T)$, and functions $f_1,\ldots, f_{\ell}\in
    L^\infty(\mu)$,\\ \notag the sequence $\Big(\frac1N \sum_{n=1}^N
    T^{s_n^g} f_1\cdot T^{2s_n^g}f_2\cdot\ldots\cdot T^{\ell
      s_n^g}f_{\ell}\Big)$ converges in $L^2(\mu)$.
  \end{enumerate}
\end{proposition}
The rest of this section will be devoted to the construction of a
sequence $(s_n)$ that satisfies conditions (s1) and (s2).
\subsection{Definition of the sequence $(s_n)$}\label{SS:s_n2}
Let $\alpha$ be any irrational number
and  $I_+$ and $I_-$ be the intervals defined in \eqref{E:I}.  The
sequence $(s_n)$ consists of the elements of a set $S$, taken in
increasing order, that is defined as follows:
\begin{align}\label{E:S'}
  S=\bigcup_{j\ge 1}\big\{ n\in\N\mid [2^{2j-1}\leq n< 2^{2j} ,\, \
  &(n^{\ell b_j}+n^{b_j})\alpha\in I_+] \ \text{ or }\\
  \notag & [2^{2j}\leq n< 2^{2j+1},\, (n^{\ell b_j}+n^{b_j})\alpha\in
  I_- ]\big\}
\end{align}
where the sequence $(b_j)$ satisfies the following conditions:
\begin{enumerate}
\item[(b1)]   $b_j\in B$, and all elements of $B$ appear infinitely
  often in the sequence $(b_j)$.

\item[(b2)] We have
$$  \lim_{j\to\infty}\sup_{N>2^{2j-1}}\left|\frac1N\sum_{n=1}^N \indicator_{I_\pm}(n^{b_j}\alpha) -\frac14
\right|=0.$$

\item[(b3)] For every system $(X,\mathcal{B},\mu,T)$, functions
  $f_1,\ldots, f_{\ell}\in L^\infty(\mu)$, and nonzero $k\in\Z$, we have
$$\lim_{j\to\infty}\sup_{N\geq2^{2j-1}}\norm{ \frac1N  \sum_{n=1}^N
  e( k(n^{\ell b_j}+n^{b_j})\alpha)\cdot T^{n^g} f_1\cdot
  T^{2n^g}f_2\cdot\ldots\cdot T^{\ell n^g}f_\ell}{L^2(\mu)}{}=0.
$$
  \end{enumerate}
  In the next subsection we construct such a sequence  $(b_j)$.

\subsection{Construction of the sequence $(b_j)$}\label{SS:b_j2}
We need two preliminary results. The first was proved in \cite{FLW2} using the 
machinery of nil-factors:
\begin{lemma}\label{L:1'}
  Let $(X,\mathcal{B},\mu, T)$ be a system, $f_1,\ldots, f_\ell \in
  L^\infty(\mu)$ and $\alpha\in \R\setminus \mathbb{Q}$. If $b,g$ are
  distinct positive integers then
  \begin{align}\label{E:main}
    \lim_{N\to\infty}\frac{1}{N}\sum_{n=1}^N & e\big((n^{{\ell
        b}}+n^{b})\alpha\big) \cdot T^{n^{g}}f_1\cdot T^{2n^{g}}f_2\cdot  \ldots \cdot
    T^{\ell n^g}f_\ell=0
  \end{align}
  where the convergence takes place in $L^2(\mu)$.
\end{lemma}
The second is the following result:
\begin{lemma}\label{L:2'}
  Let $\ell,g\in\N$. Then there exists $d(\ell,g)\in \N$ such that
  for every $d\geq d(\ell,g)$ and $\alpha\in \R\setminus \mathbb{Q}$,
  we have
  \begin{equation*}
    \lim_{N\to\infty} \sup_{S_1} \norm{ \frac1N  \sum_{n=1}^N  e\big( (n^{\ell d}+n^d)\alpha\big)\cdot
      T^{n^g} f_1\cdot  T^{2n^g}f_2\cdot\ldots\cdot
      T^{\ell n^g}f_\ell}{L^2(\mu)}{}=0,
  \end{equation*}
  where $S_1$ is the collection of all systems $(X,\mathcal{B},\mu,T)$
  and functions $f_1,\ldots, f_\ell\in L^\infty(\mu)$ with
  $\norm{f_i}{L^\infty(\mu)}{}\leq 1$ for $i=1,\ldots,\ell$.
\end{lemma}
\begin{proof}
  The main idea is to apply a Hilbert space version of van der Corput's classical inequality (see Section~\ref{SS:VDC}) several times in
  order to get an upper bound for the expression
  \begin{equation}\label{E:beforeVDC}
    \norm{ \frac1N  \sum_{n=1}^N  e\big( (n^{\ell d}+n^d)\alpha\big)\cdot
      T^{n^g} f_1\cdot  T^{2n^g}f_2\cdot\ldots\cdot
      T^{\ell n^g}f_\ell}{L^2(\mu)}{2^{\ell d -1}}
  \end{equation}
  that does not depend on the transformation $T$ or the functions
  $f_1,\ldots, f_\ell$. Such an estimate can be obtained using a
  rather standard argument, very much along the lines of the
  polynomial exhaustion technique introduced by Bergelson in
  \cite{Be1}.  To do this we shall use Lemma~\ref{L:PET1} and
  Proposition~\ref{P:PET2} in the Appendix. We get that for $d$ large
  enough (depending on $\ell$ and $g$ only), the quantity in
  \eqref{E:beforeVDC} is bounded by a constant multiple of
  \begin{equation}\label{E:afterVDC}
    \frac{1}{H_1\cdots H_{\ell d-1}}\sum_{1\leq h_i\leq H_i} \Big|\frac{1                                                                                 }{N}\sum_{n=1}^N e\big(\Delta_{h_1,\ldots,h_{\ell d-1}}(n^{\ell d}+n^d)\alpha\big)\Big|+o_{N,H_i,H_i\prec N}(1),
  \end{equation}
  where $\Delta_h(a_n)=a_{n+h}-a_n$,
  $\Delta_{h_1,\ldots,h_r}(a_n)=\Delta_{h_1}\Delta_{h_2}\cdots\Delta_{h_r}(a_n)$,
  and $o_{N,H_i,H_i\prec N}(1)$ denotes a quantity that goes to zero
  as $N, H_i \to \infty $ in a way that $H_i/N\to 0$. Notice that the
  sequence $\Delta_{h_1,\ldots,h_{\ell d-1}}(n^{\ell d}+n^d)$ is
  linear in $n$. Since $\alpha$ is irrational, letting $N\to +\infty$
  and then $H_i\to+\infty$, we get that the quantity \eqref{E:afterVDC}
  converges to $0$. This completes the proof.
\end{proof}

Using that $\frac1N \sum_{n=1}^{N} \indicator_{I_\pm}((n^{\ell b}+n^b)\alpha) \to
\frac14$ for $b\in\N$, and Lemmas~\ref{L:1'} and \ref{L:2'}, the next
result is proved in a similar fashion as
Proposition~\ref{construction}.
\begin{proposition}\label{construction'}
  There exists a sequence $(b_j)$ that satisfies conditions (b1),
  (b2), and (b3) of Section~\ref{SS:s_n2}.
\end{proposition}\label{P:2'}

\subsection{The sequence $(s_n)$ satisfies (s1) and (s2)}
The next result is proved in essentially the same way as
Proposition~\ref{P:mean} and allows us to immediately deduce
Theorem~\ref{T:multiple}.
\begin{proposition}
  Suppose that the sequence $(b_j)$ satisfies conditions (b1), (b2),
  and (b3) of Section~\ref{SS:s_n2}.  Then the sequence $(s_n)$
  defined by \eqref{E:S'} satisfies conditions (s1) and (s2) of
  Proposition~\ref{mean'}.
\end{proposition}

\section{Pointwise convergence}
We shall prove Theorem~\ref{T:main}. The argument is
similar to the one used to prove Theorem~\ref{sec:main-result-1}. However, extra complications arise
since, as is typical for pointwise results, we need to establish quantitative estimates for some
trigonometric sums.

Throughout this section we shall assume that the irrational $\alpha$
is badly approximable, that means,
  there exists a positive real number  $c$ such that for every $p\in \Z$ and $q\in \N$ we
  have $|\alpha-p/q|\geq c/q^2$. In fact, for convenience, we shall fix $\alpha$ to be the {\it golden mean}
  $(\sqrt{5}+1)/2$, in which case  the previous estimate holds with  $c=1/3$.
\subsection{A reduction}
Theorem~\ref{T:main} is a direct consequence of the following result:
\begin{theorem}\label{T:pointwise}
  Let $B$ be an arbitrary set of positive integers  and $\alpha$ be the golden mean.  Then, there exists
  an increasing sequence $(s_n)$ of integers such that
  \begin{enumerate}
  \item[(s1)] For every $b\in B$ the sequence $\big(\frac{1}{N}\sum_{n=1}^N e(s_n^b\alpha)\big)$ diverges.
  \item[(s2)] For every $g\in \mathbb N\setminus B$ the sequence
    $(s_n^g)$ is good for the pointwise ergodic theorem.
  \end{enumerate}
\end{theorem}
Our first goal is to find a more convenient condition to replace
(s2). To do this we are going to use the following lemma:
\begin{lemma}\label{standard}
  Let $(w_n)_{n\in\N}$ be a bounded sequence of complex numbers, and
  $(a_n)_{n\in\N}$ be a sequence of positive integers, such that for
  all $\gamma>1$ we have
$$
\sum_{k=1}^\infty\sup_{\beta\in\R}\left|
  \frac1{\left[\gamma^k\right]}\sum_{n=1}^{\left[\gamma^k\right]}w_n \ \! e(a_n\beta)\right|^2<+\infty.
$$
Then for every system $(X,\mathcal{B},\mu,T)$ and $f\in
L^2(\mu)$, we have
$$
\lim_{N\to\infty}\frac1N\sum_{n=1}^Nw_n\cdot
T^{a_n}f=0\qquad\text{${\mu}$-almost everywhere.}
$$
\end{lemma}
\begin{proof}
 Using the spectral theorem for unitary operators we get that
  $$
\Big\|\frac1N\sum_{n=1}^Nw_n\cdot
  T^{a_n}f\Big\|_2^2= \int \Big|\frac1N\sum_{n=1}^{N}w_n \ \! e(a_nt)\Big|^2 \ d\sigma_f(t)
  $$
holds for every $N\in\N$, where $\sigma_f$ denotes the spectral measure of the function $f$. As a consequence
$$
\Big\|\frac1N\sum_{n=1}^Nw_n\cdot
  T^{a_n}f\Big\|_2^2\leq\|f\|_2^2\cdot\sup_{t\in\R}\Big|\frac1N\sum_{n=1}^{N}w_n \ \! e(a_n t)\Big|^2.
$$
Combining this with our hypothesis, we get that if $\gamma>1$,
then
$$
\lim_{k\to+\infty}\frac1{\left[\gamma^k\right]}\sum_{n=1}^{\left[\gamma^k\right]}w_n\cdot
T^{a_n}f=0\qquad\text{$\mu$-almost everywhere for $f\in L^2(\mu)$}.
$$
(We used  that $\sum_{k=1}^\infty\|f_k\|_2<\infty$ implies $f_k\to 0$ pointwise.)
The announced result now follows from  Lemma 1.5 in \cite{RW}.
\end{proof}
Let  $I_+, I_-$ be the intervals defined by \eqref{E:I}. Given a sequence of positive integers $(b_j)$ and
 functions $\phi,\psi\colon \T\to \T$ let  $(f_n(\phi,\psi))$ be the sequence defined by \eqref{alternance}.
As in Section~\ref{sectionmean}, we define a
sequence $(s_n)$ by taking the elements of the set $S$ given by
$$
  \indicator_S(n)= f_n(\indicator_{I_+}, \indicator_{I_-})
$$
in increasing order.
\begin{proposition}\label{P:pointwise}
  Let $B$ be an arbitrary set of positive integers,  $\alpha\in \R$, and let the sequences
   $(f_n(\phi,\psi))$ and   $(s_n)$ be as above. Suppose that
  \begin{enumerate}
  \item[(s1')] For every $b\in B$, the sequence $\Big(\frac1N
    \sum_{n=1}^N e(s_n^b\alpha)\Big)$ diverges.
  \item[(s2')] For all trigonometric polynomials $\phi$ and $\psi$
    with zero integral, $\gamma>1$, and $g\in \N\setminus B$, we have
 $$
 \sum_{k\geq0}\sup_{\beta\in\R}\left|\frac1{\left[\gamma^k\right]}\sum_{n=1}^{\left[\gamma^k\right]}
   f_n(\phi,\psi) \ \! e(n^g\beta)\right|<+\infty.
  $$
\end{enumerate}
Then the sequence $(s_n)$ satisfies conditions (s1) and (s2) of
Theorem~\ref{T:pointwise}.
\end{proposition}
\begin{proof}
  It is clear that if (s1') holds then also (s1) holds.

  It remains to show that condition (s2') implies that for $g\in
  \mathbb N\setminus B$ the sequence $(s_n^g)$ is good for the  pointwise
  ergodic theorem. We start by observing that condition (s2'), combined
  with Lemma~\ref{standard}, guarantees that for every system
  $(X,\mathcal{B},\mu,T)$  we have
  \begin{equation}\label{E:S(n)}
   \lim_{N\to\infty}\frac1N\sum_{n=1}^Nf_n(\phi,\psi) \cdot T^{n^g}f=0\quad\text{${\mu}$-almost everywhere for $f\in L^\infty(\mu)$.}
  \end{equation}
  Using the approximation argument of Lemma~\ref{L:mean1} and the fact $\frac1N\sum_{n=1}^Nf_n(\phi,\psi)\to 0$ (this follows from \eqref{E:S(n)}), we conclude that \eqref{E:S(n)}
  remains true if we replace $f_n(\phi,\psi)$ by $
  f_n(\indicator_{I_+}, \indicator_{I_-})-1/4=\indicator_S(n)-1/4$. As a consequence,
  we have
  \begin{equation}\label{E:S(n)'}
    \lim_{N\to\infty}\frac1N\sum_{n=1}^N(\indicator_S(n)-1/4)\cdot T^{n^g}f=0\quad\text{${\mu}$-almost everywhere for $f\in L^\infty(\mu)$.}
  \end{equation}
 Using Bourgain's maximal inequality (\cite{Bou1})
  for the ergodic  averages $\frac1N\sum_{n=1}^NT^{n^g}f$, we can  to replace
  $L^{\infty}(\mu)$ by $L^2(\mu)$ in the preceding statement.
  Finally, using Bourgain's pointwise ergodic theorem (\cite{Bou1}, or
  see the Appendix~B of \cite{Be3} for a simpler proof) for the
  ergodic averages along $g$-th powers, we get that
$$
\lim_{N\to\infty}\frac1N\sum_{n=1}^N\indicator_S(n)\cdot
T^{n^g}f\quad\text{exists ${\mu}$-almost everywhere for  $f\in L^\infty(\mu)$}.
$$
Since the set $S$ has positive density (this follows by setting $f=1$ in
\eqref{E:S(n)'}), we conclude that
$$
\lim_{N\to\infty}\frac1N\sum_{n=1}^N T^{s_n^g}f\quad\text{exists
  ${\mu}$-almost everywhere for $f\in L^2(\mu)$}.
$$
Therefore, the sequence $(s_n^g)$ is good for the pointwise ergodic
theorem. This completes the proof.
\end{proof}
The rest of this section will be devoted to the construction of a
sequence $(s_n)$ that satisfies conditions (s1') and (s2').

\subsection{Definition of the sequence $(s_n)$}\label{SS:s_n3}
We remind the reader the set of integers $B$ is given and the irrational number $\alpha$ is the golden mean.
Let $I_+$ and $I_-$ be the intervals defined in \eqref{E:I}.  The
sequence $(s_n)$ consists of the elements of a set $S$, taken in
increasing order, that is defined as follows:
\begin{equation}\label{E:pointS}
  S=\bigcup_{j\ge 1}\left\{ n\in\N\mid [2^{2j-1}\leq n< 2^{2j}
    ,\, \
    n^{b_j}\alpha\in I_+]\text{ or }[2^{2j}\leq  n< 2^{2j+1},\, n^{b_j}\alpha\in I_- ]\right\}
\end{equation}
where the sequence of integers $(b_j)$ satisfies the following
conditions:
\begin{enumerate}
\item[(b1)] $b_j\in B$, and all elements of $B$ appear infinitely
  often in the sequence $(b_j)$.

\item[(b2)] We have
$$  \lim_{j\to\infty}\sup_{N>2^{2j-1}}\left|\frac1N\sum_{n=1}^N \indicator_{I_\pm}(n^{b_j}\alpha) -\frac14
\right|=0.$$

\item[(b3)] For every nonzero $m\in\Z$, and
  $g\in \N\setminus B$, we have
$$
\lim_{j\to\infty}\sup_{N\geq2^{2j-1}} \sup_{\beta\in
  \R}\left|\frac1{N^{1-\eta(N)}} \sum_{n=1}^{N} e(m n^{b_j}\alpha +
  n^g\beta)\right|<+\infty
$$
where $\eta(N)=(\log_2(N))^{-1/2}$.
\end{enumerate}
In the next subsection we construct such a sequence $(b_j)$.

  \subsection{Construction of the sequence $(b_j)$}\label{SS:b_j3}
  The key ingredient in the construction is the following exponential
  sum estimate:
  \begin{proposition}\label{P:key}
    Let $\alpha$ be the golden mean, and
    $(\eta(N))$ be a sequence of real numbers which tends to zero at
    infinity. For every
    $b,g\in\N$ with $b\neq g$, and nonzero $m\in \Z$,   there exists
    $N_0=N_0(b,g,m)$, such that if $N>N_0$ then
$$
\sup_{\be\in\R}\left|\sum_{n=1}^{N} e(m n^b\al+n^g\be)\right|\leq
N^{1-\eta(N)}.
$$
\end{proposition}
\begin{proof}
  This is an immediate consequence of Lemmas~\ref{gsb} and \ref{bsg}
  in the Appendix.
\end{proof}

Using that $ \frac1N \sum_{n=1}^{N} \indicator_{I_\pm}(n^b\alpha) \to
\frac14$ for $b\in\N$, and Proposition~\ref{P:key}, the next result is
proved in essentially the same way as Proposition~\ref{construction}
(the argument is actually somewhat simpler in this case, since we
have already combined the estimates dealing with $b<g$ and $b>g$ into
a single estimate).
\begin{proposition}
  \label{sec:constr-sequ-b-j}
  There exists a sequence $(b_j)$ that satisfies conditions (b1),
  (b2), and (b3) of Section~\ref{SS:s_n3}.
\end{proposition}

\subsection{The sequence $(s_n)$ satisfies conditions (s1) and (s2)}
The goal of this section is to prove the following proposition, that
allows us to immediately deduce Theorem~\ref{T:main}.
\begin{proposition}\label{P:s1s2}
  Suppose that the sequence $(b_j)$ satisfies conditions (b1), (b2),
  and (b3) of Section~\ref{SS:s_n3}. Then the sequence $(s_n)$ defined
  by \eqref{E:pointS} satisfies conditions (s1) and (s2) of
  Theorem~\ref{T:pointwise}.
\end{proposition}
Before starting the proof of Proposition~\ref{P:s1s2}, let us gather
some useful properties that the sequence $\eta(N)=(\log_2(N))^{-1/2}$
satisfies:
\begin{itemize}
\item[($\eta$1)] The sequence $\left(N^{-\eta(N)}\right)$ is
  decreasing.
\item[($\eta$2)] For every $\gamma>1$, we have
  $\displaystyle\sum_{k\geq0}\left[\gamma^k\right]^{-\eta\left(\left[\gamma^k\right]\right)}<+\infty$.
\item[($\eta$3)] If we define
$$
\rho(N)=\frac1N\sum_{i=1}^{\left[\log_2N\right]}2^{i(1-\eta(2^i))},
$$
then for every $\gamma>1$, we have
$\displaystyle\sum_k\rho\left(\left[\gamma^k\right]\right)<+\infty$.
\end{itemize}

The first property is obvious.

To check the second property notice that
$\eta\left(\left[\gamma^k\right]\right)\sim c/\sqrt{k}$, so
$\left[\gamma^k\right]^{-\eta\left(\left[\gamma^k\right]\right)}=O\left(\gamma^{-c\sqrt{k}}\right)$
for some constant $c>0$.

To check the third property notice
that $$\sum_{i=1}^{l}2^{i(1-\eta(2^i))}=\sum_{i=1}^l2^{i-\sqrt{i}}=\sum_{i=1}^{[l/2]-1}2^{i-\sqrt{i}}+\sum_{i=[l/2]}^l2^{i-\sqrt{i}}\leq2^{l/2}+2^{-\sqrt{l/2}}2^{l+1},
$$
hence
$$
\sum_{i=1}^{l}2^{i(1-\eta(2^i))}\leq2^{l-\sqrt{l/2}+2}.
$$
It follows that
$$
\rho\left(\left[\gamma^k\right]\right)\leq\frac1{\left[\gamma^k\right]}2^{\log_2\left(\left[\gamma^k\right]\right)-\sqrt{\log_2\left(\left[\gamma^k\right]\right)/2}+2
}=O\left(2^{-\sqrt{k(\log_2{\gamma})/2}}\right),
$$
which implies ($\eta$3).


\begin{lemma}\label{L:b3}
  Suppose that the sequence $(b_j)$ satisfies condition (b3) of
  Section~\ref{SS:s_n3}.  If $\phi$ and $\psi$ are two trigonometric
  polynomials with zero integral, let $(f_n(\phi,\psi))$ be the
  sequence defined by \eqref{alternance}.  Then for every $\gamma>1$ and for every
  $g\in \N\setminus B$,
  we have
 $$
 \sum_{k\geq0}\sup_{\beta\in\R}\left|\frac1{\left[\gamma^k\right]}\sum_{n=1}^{\left[\gamma^k\right]}
   f_n(\phi,\psi)\ \! e(n^g\beta)\right|<+\infty.
 $$
\end{lemma}
\begin{proof}
  For brevity we shall write $f_n$ instead of $f_n(\phi,\psi)$.  Since
  $(b_j)$ satisfies condition (b3), there exists a positive constant
  $C=C(\phi,\psi,g)$ such that for every
  large enough $j$ and  $N\geq2^{2j-1}$, we have
  \begin{equation}\label{E:xi}
    \sup_{\beta\in\R}\left|\sum_{n=1}^N\xi(n^{b_j}\alpha)\ \! e\left(n^g\beta\right)\right|\leq C N^{1-\eta(N)},
  \end{equation}
  where $\xi$ is either $\phi $ or $\psi$.  We shall use this estimate
  to find an upper bound for the averages $$
  \frac{1}{N}\sum_{n=1}^Nf_n \ \! e(n^g\beta).
 $$
 We start by noticing that
$$
\frac1{2^{2j}}\sum_{n=2^{2j}+1}^{2^{2j+1}}f_n \ \! e(n^g\beta) =
2\frac1{2^{2j+1}}\sum_{n=1}^{2^{2j+1}}\psi(n^{b_j}\alpha)\ \! e(n^g\beta)-
\frac1{2^{2j}}\sum_{n=1}^{2^{2j}}\psi(n^{b_j}\alpha)\ \! e(n^g\beta).
  $$
  We also get a similar estimate with $2j+1$ is in place of $2j$ and
  $\phi$ in place of $\psi$.  It follows from \eqref{E:xi} that there
  exists $j_0\geq0$ (depending on $\phi$, $\psi$ and $g$), such that
  for every $j\geq j_0$, we have
  \begin{equation}\label{E:f_n}
    \sup_{\beta\in\R} \left|\frac1{2^j}\sum_{n=2^j+1}^{2^{j+1}}f_n\ \! e(n^g\beta) \right|\leq C\,  (2^j)^{-\eta(2^j)}.
  \end{equation}
  Now consider a large enough integer $N$, then $N\in(2^j,2^{j+1}]$
  for some $j\geq j_0$. We split the sum between 1 and $N$ into
  several pieces
$$
\sum_{n=1}^{N}\cdot=\sum_{n=1}^{2^{j_0}}\cdot+\sum_{i=j_0}^{j-1}\sum_{n=2^i+1}^{2^{i+1}}\cdot+\sum_{n=2^j+1}^{N}\cdot,
$$
and we get the following estimate
\begin{equation}\label{E:e1}
  \left|\frac1N\sum_{n=1}^{N}f_n\ \! e(n^g\beta)\right| \leq C\frac{2^{2j_0}}{N}+\frac1N\sum_{i=j_0}^{j-1}\left|\sum_{n=2^i+1}^{2^{i+1}}f_n\ \! e(n^g\beta)\right|+\left|\frac1N\sum_{n=2^j+1}^{N}f_n\ \! e(n^g\beta)\right|.
\end{equation}
By \eqref{E:f_n} we have
\begin{equation}\label{E:e2}
  \frac1N\sum_{i=j_0}^{j-1}\left|\sum_{n=2^i+1}^{2^{i+1}}f_n\ \! e(n^g\beta)\right|\leq C\frac1N\sum_{i=1}^{[\log_2 N]}2^{i(1-\eta(2^i))}.
\end{equation}
Moreover, since
$$
\left|\frac1N\sum_{n=2^j+1}^{N}f_n\ \! e(n^g\beta)\right|\leq\left|\frac1N\sum_{n=1}^{N}\xi(n^{b_j}\alpha)\ \!
  e(n^g\beta)\right|+\left|\frac1{2^j}\sum_{n=1}^{2^j}\xi(n^{b_j}\alpha)\ \!
  e(n^g\beta)\right|,
 $$
 where $\xi$ is either $\phi $ or $\psi$, we get using \eqref{E:xi}
 and property $(\eta1)$ that
 \begin{equation}\label{E:e3}
   \left|\frac1N\sum_{n=2^j+1}^{N}f_n\ \! e(n^g\beta)\right|\leq C \left(N^{-\eta(N)}+ (2^j)^{-\eta(2^j)}\right)\leq 2C \, (N/2)^{-\eta(N/2)}.
 \end{equation}
 Combining equations~\eqref{E:e1}, \eqref{E:e2}, and \eqref{E:e3}, and
 using properties $(\eta1)$, $(\eta2)$, and $(\eta3)$, we get the
 advertised result.
\end{proof}
\begin{proof}[Proof of Proposition~\ref{P:s1s2}]
 Let $(s_n)$  be the sequence  defined by
  \eqref{E:pointS}.
   By Proposition~\ref{P:pointwise}, it suffices to
  verify properties (s1') and (s2') mentioned there.  Using
  Proposition~\ref{P:1}, we see that properties (b1) and (b2) give
  property (s1').  Also, using Lemma~\ref{L:b3}, we see that property
  (b3) gives property (s2'). This completes the proof.
\end{proof}

\section{Appendix}
We prove some results that were used in the main part of the article.
\subsection{Van der Corput's lemma} \label{SS:VDC} The following is a Hilbert space
version of a classical elementary estimate of van der Corput. It appears in
a form similar to the one stated below in \cite{Be1}.
\begin{lemma} \label{L:VDC} Let $v_1,\ldots, v_N$ be vectors
  of a Hilbert space with $\|v _i\| \leq 1$ for $i=1,\ldots,
  N$. Then for every integer  $H$ between $1$ and $N$ we have
$$
\norm{\frac{1}{N}\sum_{n=1}^N v_n}{}{2}\leq \frac{2}{H}+
\frac{4}{H}\sum_{h=1}^{H-1}\Big|\frac{1}{N} \sum_{n=1}^{N-h}\langle
v_{n+h},v_n\rangle\Big|.
$$
\end{lemma}
An immediate corollary of the preceding lemma is the following:
\begin{corollary}\label{C:VDC}
  Let $v_1,\ldots, v_N$ be vectors of a Hilbert space with $\|v _i\|\leq 1$ for $i=1,\ldots, N$.  Then for every integer  $H$ between $1$ and $N$ we have
$$
\norm{\frac{1}{N}\sum_{n=1}^N v_n}{}{2}\leq
\frac{4}{H}\sum_{h=1}^{H}\Big|\frac{1}{N} \sum_{n=1}^{N}\langle
v_{n+h},v_n\rangle\Big| +o_{N,H,H\prec N}(1),
$$
where $o_{N,HH\prec N}(1)$ denotes a quantity that goes to zero as
$N, H \to \infty $ in a way that $H/N\to 0$.
\end{corollary}
\subsection{Dyadic intervals}
The next lemma allows us, under suitable assumptions,  to concatenate dyadic pieces of sequences and create a new sequence with average zero.
\begin{lemma}\label{L:Wierdl}
  Let $(u_{n,j})_{n,j\in\N}$ be a family of complex numbers that satisfy
$$
\lim_{j\to\infty} \sup_{N>2^{j}}\left| \frac{1}{N} \sum_{n=1}^N
  u_{n,j}\right|=0.
$$
Define the sequence $(u_n)$  by
$$
u_n=u_{n,j} \ \text{ if } \ 2^{j}\leq n<2^{j+1}.
$$
Then
$$
\lim_{N\to\infty}\frac{1}{N}\sum_{n=1}^N u_n=0.
$$
\end{lemma}
\begin{proof}

  Let $\varepsilon>0$.  We define
$$
\varepsilon(j)=\sup_{N>2^{j}}\left|\frac1N\sum_{n=1}^N u_{n,j}
\right|
$$
and
$$
\varepsilon'(j)=\sup_{k\geq j} \varepsilon(k).
$$
By our assumption, there exists $j_0$ such that
$\varepsilon'(j)<\varepsilon$ for every $j\geq j_0$.

We start by noticing that
  $$
  \frac1{2^{j}}\sum_{n=2^{j}+1}^{2^{j+1}} u_n = 2\,
  \frac1{2^{j+1}}\sum_{n=1}^{2^{j+1}}u_{n,j}-
  \frac1{2^{j}}\sum_{n=1}^{2^{j}}u_{n,j}.
  $$
   Therefore,  for every
  $j\in\N$ we have
  \begin{equation}\label{c2}
    \left|\frac1{2^j}\sum_{n=2^j+1}^{2^{j+1}}u_n\right|\leq 3\, \varepsilon(j).
  \end{equation}
  Now suppose that $N>2^{j_0}$, then $N\in(2^j,2^{j+1}]$ for some
  $j\geq j_0$. We split the sum between 1 and $N$ into several pieces
$$
\sum_{n=1}^{N}\cdot=\sum_{n=1}^{2^{j_0}}\cdot+\sum_{i=j_0}^{j-1}\sum_{n=2^i+1}^{2^{i+1}}\cdot+\sum_{n=2^j+1}^{N}\cdot,
$$
in order to get the following upper bound
  $$
  \left|\frac1N\sum_{n=1}^{N}u_n\right| \leq \left|\frac1N
    \sum_{n=1}^{2^{j_0}}u_n\right|+\frac1N\sum_{i=j_0}^{j-1}\left|\sum_{n=2^i+1}^{2^{i+1}}u_n\right|+
  \left|\frac1N\sum_{n=2^j+1}^{N}u_n\right|.
  $$
  Using (\ref{c2}), we get
  $$
  \frac1N\sum_{i=j_0}^{j-1}\left|\sum_{n=2^i+1}^{2^{i+1}}u_n\right|
  \leq\frac1N\sum_{i=j_0}^{j-1}3\, \varepsilon(i)\, 2^i\leq3\,
  \varepsilon'(j_0)\frac1N\sum_{i=j_0}^{j-1}2^i\leq3\varepsilon'(j_0).
  $$
  We also have
  $$
  \frac1N\sum_{n=2^j+1}^{N}u_n=\frac1N\sum_{n=1}^{N}u_{n,j}-
  \frac1N\sum_{n=1}^{2^j}u_{n,j},$$ which gives
  $$
  \left|\frac1N\sum_{n=2^j+1}^{N}u_n\right|\leq
  \left|\frac1N\sum_{n=1}^{N}u_{n,j}\right|+
  \left|\frac1{2^j}\sum_{n=1}^{2^j}u_{n,j}\right|\leq
  2\,\varepsilon'(j).
  $$
  Putting these estimates together we get
  $$
  \left|\frac1N\sum_{n=1}^{N} u_n \right| \leq \left|\frac1N
    \sum_{n=1}^{2^{j_0}}u_n\right|+ 5\varepsilon'(j_0).  $$
    By taking $N$ large enough we can make sure that
  the right hand side becomes less than $6\, \varepsilon$.  This
  completes the proof.
\end{proof}

 \subsection{Exponential sum estimates}
 We shall establish two exponential sum estimates.  The first gives non-trivial
 power type savings when one deals with exponential sums involving polynomials with leading
 coefficient an integer multiple of the golden mean. Let us recall the
 ``bad approximation property'' of the golden mean $\alpha$
 \begin{equation}\label{E:bad}
   \text{ For all non-zero integers $q$ we have $d(q\alpha,\Z)\geq 1/(3q)$}.
 \end{equation}

 \begin{lemma}\label{gsb}
   Let $\al$ be the golden mean and $b\in\N$. There exists $C=C(b)>0$
   such that for every $m,N\in\N$
   we have
   \begin{equation}\label{E:b}
    \sup_{P\in\R[X], \deg(P)<b}  \left|\sum_{n=1}^{N} e(m n^b\al+P(n))\right|\leq C\,m^{2^{1-b}} N^{1-4^{1-b}}.
   \end{equation}
 \end{lemma}
 \begin{proof}
   We use an induction on $b$ . For $b=1$, we have
$$\left|\sum_{n=1}^{N} e(m n\al)\right|\leq\frac2{|e(m\al)-1|}=\frac1{|\sin(\pi m\al)|}\leq\frac1{2\text{d}(m\al,\Z)}\leq\frac{3m}2,$$
by the bad approximation property \eqref{E:bad}.

Suppose that the estimate \eqref{E:b} holds for the integer $b$. We
are going to show that it also holds for the integer $b+1$.  Let us
define
$$
S(m,N,b)=\sup_{P\in\R[X], \deg(P)<b}\left|\sum_{n=1}^{N} e(m
  n^b\al+P(n))\right|.
$$
From van der Corput's inequality (Lemma~\ref{L:VDC}), we deduce that
for every integer $H$ between 1 and $N$, we have
\begin{equation}\label{E:a1}
  S(m,N,b+1)^2\leq\frac{2N^2}{H}+\frac{4N}{H}\sum_{h=1}^{H-1}S(mh(b+1),N-h,b).
\end{equation}
The induction hypothesis gives that for some constant $C=C_b$ we have
$$
S(mh(b+1),N-h,b)\leq C\, (mh)^{2^{1-b}}(N-h)^{1-4^{1-b}},
$$
and so for $h$ between 1 and $H-1$ we have
\begin{equation}\label{E:a2}
  S(mh(b+1),N-h,b)\leq C\,m^{2^{1-b}}(H-1)^{2^{1-b}}N^{1-4^{1-b}}.
\end{equation}
Combining \eqref{E:a1} and \eqref{E:a2} we get
\begin{equation}\label{E:step}
  S(m,N,b+1)^2\leq 4\, C\,m^{2^{1-b}}\left(N^2H^{-1}+N^{2-4^{1-b}}(H-1)^{2^{1-b}}\right).
\end{equation}
Choosing $H=\left[N^{2^{1-2b}}\right]+1$ gives
$$
N^2H^{-1}\leq N^{2(1-4^{-b})}, \qquad N^{2-4^{1-b}}(H-1)^{2^{1-b}}\leq
N^{2(1-2\cdot 4^{-b}+2\cdot 8^{-b})}\leq N^{2(1-4^{-b})}.
$$
Using this together with \eqref{E:step} we find that
 $$
 S(m,N,b+1)^2\leq 4\, C\,m^{2^{1-b}} N^{2(1-4^{-b})}.
 $$
 Taking square roots establishes \eqref{E:b} for $b+1$ . This
 completes the induction and the proof.
\end{proof}
The second lemma gives non-trivial power type savings for exponential sums involving
polynomials
that have an integer multiple of the golden mean as a non-leading
(non-constant) coefficient. Its proof is a simplification of an
argument that appears in \cite{BKQW}.
\begin{lemma}\label{bsg}
  Let $\al$ be the golden mean and $g\in\N$. For every $\e>0$, there
  exists $C=C(\e,g)<+\infty$ such that, for every $\be\in\R$,
  $b\in\N$ with $b<g$,  $N\in\N$, and  nonzero $m\in\Z$ with
  $|m|\leq N^{2^{-g-2}}$, we have
  \begin{equation}\label{est}
    \left|\sum_{n=1}^{N} e(m n^b\al+n^g\be)\right|\leq C\,N^{1+\e-2^{-2g-1}}.
  \end{equation}
\end{lemma}
\begin{proof}
  The proof proceeds as follows: If $\beta$ is not well approximable
  by rationals in a way to be made precise below, then classical estimates of Weyl immediately give the
  advertised estimate.  If $\beta$ is well approximated by rationals,
  using partial summation we can replace $\beta$ by a rational (up to
  a small error), and reduce the problem to studying an exponential
  sum involving a polynomial that has an integer multiple of $\alpha$
  as leading coefficient. In this case, again the classical estimates
  of Weyl give the advertised result.

  So let us first recall Weyl's classical estimate (see e.g. \cite{Va}). For every $k\in\N$ and $\e>0$, there exists a constant $C$
  satisfying the following property: for every $N\in\N$, $\be\in\R$,   and relatively
  prime $r,s\in\N$ with $|\be-r/s|<1/s^2$, and for every real polynomial
  $P(x)$ with leading coefficient $\be x^k$, we have
 $$
 \left|\sum_{n=1}^{N} e(P(n))\right|\leq C
 N^{1+\e}\left(\frac1s+\frac1N+\frac{s}{N^k}\right)^{1/2^{k-1}}.
 $$

 We fix $\be\in\R$, $b\in\N$ with $b<g$, $N\in\N$, nonzero $m\in\Z$ with
 $|m|\leq N^{2^{-g-2}}$, and we define $\ga=2^{-g-2}$.

 By Dirichlet's principle, there exist $r,s\in\N$, relatively prime,
 such that $s\leq N^{g-\ga}$ and
$$
\left|\be-\frac{r}{s}\right|\leq\frac1{N^{g-\ga}s}.
$$
We distinguish two cases: either $s>N^\ga$ (bad approximation) or
$s\leq N^\ga$ (good approximation).

{\bf Case 1.} Suppose that $s>N^\ga$. By Weyl's estimate, we have
$$
\left|\sum_{n=1}^{N} e(mn^b\al+n^g\be)\right|\leq C(\e,g)
N^{1+\e}(N^{-\ga}+N^{-1}+N^{-\ga})^{2^{1-g}},
$$
which implies the estimate (\ref{est}), since $\ga=2^{-g-2}$.

{\bf Case 2.} Suppose now that $s\leq N^\ga$.  By Dirichlet's
principle, there exist $t,u\in\N$, relatively prime, such that $u\leq
N^{b-1/2}$ and
\begin{equation}\label{ouest}
  \left|ms^b\al-\frac{t}{u}\right|\leq\frac1{N^{b-1/2}u}.
\end{equation}
The bad approximation property of $\alpha$ mentioned in \eqref{E:bad}
gives that $mus^b\geq \frac13N^{b-1/2}$. Since $s\leq N^\ga$ and
$|m|\leq N^{\ga}$ we have $u\geq\frac13 N^{b-1/2-\ga b-\ga}$.

Consider now an integer $M$ between $N^{1-\ga}$ and $N$.  We are going
to compare the sums $\sum_{n\leq M} e(mn^b\al+n^g\be)$ with the sums
$\sum_{n\leq M} e(mn^b\al+\frac{r}{s}n^g )$ that are easier to
estimate. Let us first estimate the second sum. We have
$$
\sum_{n\leq M} e\left(mn^b\al+\frac{r}{s} n^g\right)=\sum_{j=1}^{s}\
\sum_{i\geq0,\,si+j\leq M}e\left(m(si+j)^b\al+\frac{r}{s} j^g\right),
$$
hence
\begin{equation}\label{sud}
  \left|\sum_{n\leq M} e\left(mn^b\al+\frac{r}{s} n^g\right)\right|\leq\sum_{j=1}^{s}\left|\sum_{i\geq0,\,si+j\leq M}e\left(m(si+j)^b\al\right)\right|.
\end{equation}
By Weyl's estimate and (\ref{ouest}), we have
$$
\left|\sum_{i\geq0,\,si+j\leq M}e\left(m(si+j)^b\al\right)\right|\leq
C(\e,b)
\left(\frac{M}{s}\right)^{1+\epsilon}\left(\frac1u+\frac{s}{M}+u\left(\frac{s}{M}\right)^b\right)^{2^{1-b}}.
$$
Using that
$$
N^{1-\ga}\leq M\leq N\ ,\quad1\leq s\leq N^\ga,
\quad\text{and}\quad\frac13N^{b-1/2-\ga b-\ga}\leq u\leq N^{b-1/2},
$$
we obtain
\begin{multline*}
  \left|\sum_{i\geq0,\,si+j\leq
      M}e\left(m(si+j)^b\al\right)\right|\leq \\C(\e,b)
  \frac{N^{1+\e}}s \left(N^{-b+1/2+\ga b+\ga}+N^{-1+2\ga}+N^{-1/2+2\ga
      b}\right)^{2^{1-b}}.
\end{multline*}
The term $N^{-1/2+2\ga b}$ is dominant, and is bounded by
$N^{-1/4}$. It follows that
$$
\left|\sum_{i\geq0,\,si+j\leq M}e\left(m(si+j)^b\al\right)\right|\leq
C(\e,b) \frac{N^{1+\e-2^{-b-1}}}s.
$$
Since the integer $g$ is fixed and $b<g$ we have
$$
\left|\sum_{i\geq0,\,si+j\leq M}e\left(m(si+j)^b\al\right)\right|\leq
C(\e,g) \frac{N^{1+\e-2^{-g-1}}}s.
$$
In conjunction with (\ref{sud}) this gives
\begin{equation}\label{nord}
  \left|\sum_{n\leq M} e\left(mn^b \al+\frac{r}{s} n^g\right)\right|\leq C(\e,g)
  N^{1+\e-2^{-g-1}}.
\end{equation}

We come back to our main goal of estimating the sums $\sum_{n\leq M}
e(mn^b\al+n^g\be)$.  We are going to use summation by parts. We set
$S(M)=\sum_{n\leq M} e\left(mn^b \al+\frac{r}{s} n^g\right)$ and
notice that
\begin{multline*}
  \left|\sum_{n\leq N} e(mn^b \al+n^g \be)\right|\leq
  N^{1-\ga}+\left|\sum_{N^{1-\ga}<n\leq N} e(mn^b \al+n^g \be)\right|
  \\ \leq N^{1-\ga}+\left|\sum_{N^{1-\ga}<n\leq N} (S(n)-S(n-1))\,
    e\left(n^g(\be-\frac{r}s)\right)\right|.
\end{multline*}
We have
$$
\left|e\left(n^g(\be-\frac{r}s)\right)-e\left((n+1)^g(\be-\frac{r}s)\right)\right|\leq
C n^{g-1}\left|\be-\frac{r}s\right|,
$$
where the constant $C$ does not depend on $\be$ because $\be-r/s$ is
uniformly bounded.  We know that for $n\leq N$ we have $
n^{g-1}\left|\be-r/s\right|\leq N^{-1+\ga}$. So using partial
summation, we obtain
\begin{multline*}
  \left|\sum_{n\leq N} e(mn^b \al+n^g \be)\right|\leq\\
  N^{1-\ga}+\left|S\left([N^{1-\ga}]+1\right)\right|+|S(N)|+C
  \sum_{N^{1-\ga}<n\leq N}|S(n)|N^{-1+\ga}.
\end{multline*}
Using (\ref{nord}) we conclude that
$$
\left|\sum_{n\leq N} e(mn^b \al+n^g \be)\right|\leq C(\e,g)
N^{1+\e-2^{-g-1}+\ga}.
$$
Recalling that $\ga=2^{-g-2}$, we derive an estimate stronger than
(\ref{est}). This completes the proof.
\end{proof}

\subsection{The PET induction argument.} We give the details needed to
complete the proof of Lemma~\ref{L:2'}.  The next result follows
immediately from Corollary~\ref{C:VDC}:
\begin{lemma}\label{L:PET1}
  Consider a family of  integer polynomials $\{p_1,\ldots,p_k\}$ and an
  integer polynomial $p$, all of them having zero constant term. Let
  $\{q_1,\ldots,q_{k'}\}$ be the family of distinct integer
  polynomials that is defined using the following operation: we start with
  the family of polynomials
 $$
 p_1(n+h)-p_1(h)-p(n),\ldots,p_k(n+h)-p_k(h)-p(n),
 p_1(n)-p(n),\ldots,p_k(n)-p(n),
 $$
 and we remove polynomials that are identically zero and repetitions
 of polynomials.
 Then for every system $(X,\mathcal{B},\mu,T)$, and sequence of
 complex numbers $(u_n)$ with $\norm{u_n}{\infty}{} \leq 1$, we have
 \begin{multline*}
   \sup_{f_1,\ldots,f_k} \left\|\frac1N\sum_{n=1}^N u_n
     \,T^{p_1(n)}f_1\cdot\ldots\cdot T^{p_k(n)}f_k\right\|_2^2\leq\\
   \frac4H\sum_{h=1}^H\sup_{f_1,\ldots,f_{k'}} \left\|\frac1N\sum_{n=1}^N
     u_{n+h} \bar{u}_n \, T^{q_1(n)}f_1\cdot\ldots\cdot
     T^{q_{k'}(n)}f_{k'}\right\|_2 +o_{N,H,H\prec N}(1),
 \end{multline*}
 where the supremums are taken over families of functions in $L^\infty(\mu)$
 bounded by $1$.
\end{lemma}
Let $\mathcal{P}$ be a family of non-constant integer polynomials with
zero constant term. The maximum degree of the polynomials is called
the \emph{degree} of the polynomial family and we denote it by $d$.
Let $\mathcal{P}_i$ be the subfamily of polynomials of degree $i$ in
$\mathcal{P}$. We let $w_i$ denote the number of distinct leading
coefficients that appear in the family $\mathcal{P}_i$. The vector
$(d,w_d,\ldots,w_1)$ is called the \emph{type} of the polynomial
family $\mathcal{P}$.  We use an induction scheme, often called
PET induction (Polynomial Exhaustion Technique), on types of
polynomial families that was introduced by Bergelson in \cite{Be1}.
We order the set of all possible types lexicographically, this means,
$(d,w_d,\ldots,w_1)>(d',w_{d'}',\ldots,w_1')$ if and only if in the
first instance where the two vectors disagree the coordinate of the
first vector is greater than the coordinate of the second vector.
\begin{proposition}[{\bf Bergelson's PET \cite{Be1}}]\label{P:PET2}
  Let $\{p_1,\ldots,p_k\}$ be a family of non-constant integer
  polynomials with zero constant term.  After applying finitely many
  times the operation defined in Lemma~\ref{L:PET1} (for good choices
  of the auxiliary polynomial $p$ at each step) it is possible to
  obtain the empty family of polynomials.
\end{proposition}
\begin{proof}
  Let $p_{min}$ be any member of the family $\{p_1,\ldots,p_k\}$ that
  has minimal degree. Notice that after applying the operation defined
  in Lemma~\ref{L:PET1} for $p=p_{min}$, we obtain a new family of
  polynomials that has type strictly less than the type of the family
  $\{p_1,\ldots,p_k\}$.  The result now follows using induction on the
  type of the polynomial family.
\end{proof}
The following two examples illustrate how a typical PET induction
argument works:
\begin{example}\label{Ex1}
  Suppose that we start with the family of polynomials
  $\mathcal{P}_0=\{n,2n\}$ that has type $(1,2)$. Applying the
  operation defined in Lemma~\ref{L:PET1} with $p(n)=n$ we obtain the
  family $\mathcal{P}_1=\{n\}$ that has type $(1,1)$. After one more
  application of the operation we obtain an empty family of
  polynomials.
\end{example}
\begin{example}
  Suppose that we start with the family of polynomials
  $\mathcal{P}_0=\{n^2,2n^2\}$ that has type $(2,2,0)$. Applying
  successively the operation defined in Lemma~\ref{L:PET1} we obtain
  the following families of polynomials: Using $p(n)=n^2$ we get the
  family
$$
\mathcal{P}_1=\{2 nh_1, n^2+4nh_1,n^2\}
$$
that has type $(2,1,1)$. Using $p(n)=2nh_1$ we get the family
$$
\mathcal{P}_2=\{n^2+2n(h_1+h_2),n^2+2n(h_2-h_1),n^2+2nh_1,n^2-2nh_1\}
$$
that has type $(2,1,0)$. Using $p(n)=n^2$ we get the family
\begin{align*}
\mathcal{P}_3=\{2n(h_1+h_2+h_3)&,2n(h_3+h_2-h_1),2n(h_1+h_3),\\
&2n(h_3-h_1),2n(h_1+h_2),
2n(h_2-h_1),2nh_1,-2nh_1\}
\end{align*}
that has type at most $(1,8)$ (actually equal to $(1,8)$ for most
values of $h_1,h_2,h_3$).  The last family consists of linear
polynomials and can be dealt as in Example~\ref{Ex1}. After $8$ more
operations we arrive to an empty family of polynomials.
\end{example}

\end{document}